\documentclass{amsart}

\usepackage{amsfonts,amsmath,amsthm,amssymb,amscd,latexsym,enumerate,etoolbox,mathrsfs,comment}

\usepackage[all,cmtip,pdf]{xy}
\usepackage{thmtools, thm-restate}
\usepackage{tikz-cd}
\usepackage{tikz}
\usetikzlibrary{calc}

\newcommand{\R}{\mathbb{R}}
\newcommand{\C}{\mathbb{C}} 
\newcommand{\N}{\mathbb{N}}

\newcommand{\Z}{{\mathbb Z}}
\newcommand{\Q}{{\mathbb Q}}

\renewcommand{\phi}{\varphi}

\theoremstyle{plain}
    \newtheorem{theorem}{Theorem}[section]
    
    \newtheorem{corollary}[theorem]{Corollary}
    \newtheorem{proposition}[theorem]{Proposition}
    \newtheorem{conjecture}[theorem]{Conjecture}
\theoremstyle{definition}

    \newtheorem{remark}[theorem]{Remark}
    
\theoremstyle{remark}

\begin{document}

\title{A counterexample to the HK-conjecture that is principal}
\author{Robin J. Deeley}
\address{Robin J. Deeley,   Department of Mathematics,
University of Colorado Boulder
Campus Box 395,
Boulder, CO 80309-0395, USA }
\email{robin.deeley@colorado.edu}
\thanks{This work was partially supported by NSF Grant DMS 2000057.}

\begin{abstract}
Scarparo has constructed counterexamples to Matui's HK-conjecture. These counterexample and other known counterexamples are essentially principal but not principal. In the present paper, a counterexample to the HK-conjecture that is principal is given. Like Scarparo's original counterexample, our counterexample is the transformation groupoid associated to a particular odometer. However, the relevant group is the fundamental group of a flat manifold (and hence is torsion-free) and the associated odometer action is free. The examples discussed here do satisfy the rational version of the HK-conjecture.
\end{abstract}

\maketitle

\section*{Introduction}
Matui's HK-conjecture \cite{MR3552533} predicts a strong relationship between the homology and $K$-theory of an important class of groupoids (the precise statement is given below). There are counterexamples to this conjecture in the essentially principal case. The first counterexample is due to Scarparo \cite{Scarparo} and a stronger counterexample (due to Ortega and Scarparo) can be found in \cite{Scarparo2}. On the other hand, there have been a number of positive results, starting with Matui's original work \cite{MR3552533}, also see \cite{Rufus, MR4030921, MR4170644, PV1, MR4052913}. In particular, there has been quite a bit of success verifying the conjecture for particular classes of principal  (rather than essentially principal) groupoids, see in particular \cite[Corollary C]{Rufus} and \cite[Remark 3.5]{PV1}.

Nevertheless, the goal of this paper is the construction of a counterexample to Matui's HK-conjecture that is principal (rather than just essentially principal). It is worth noting that our examples do satisfy the rational version of the conjecture. 

I will now state the HK-conjecture and outline the construction of the counterexample. The reader unfamiliar with the various terms used below can see Section \ref{secPrelim} for precise definitions. The statement of the HK-conjecture is as follows:
\begin{conjecture}
Suppose that $\mathcal{G}$ is a second countable, \'etale, (essentially) principal, minimal, ample groupoid. Then
\[ K_*(C^*_r(\mathcal{G})) \cong \bigoplus_{i} H_{*+2i}(\mathcal{G}) \] 
where $K_*(C^*_r(\mathcal{G}))$ is the $K$-theory of the reduced groupoid $C^*$-algebra of $\mathcal{G}$ and $H_{*}(\mathcal{G})$ is the homology of $\mathcal{G}$.
\end{conjecture}

Like Scarparo's counterexample \cite{Scarparo}, the counterexample in the present paper is obtained from an odometer, see \cite[Section 2]{Scarparo} and the references therein for more on odometer actions. Unlike in \cite{Scarparo} the relevant group is torsion-free. The starting point is a flat manifold, $Y$, with an expanding endomorphism $g: Y \rightarrow Y$ in the sense of Shub \cite{ShubExp}. By \cite[Proposition 3]{ShubExp}, $g$ is an $n$-fold covering map (for some $n\ge 2$) and one obtains a chain of finite index subgroups

\[
\pi_1(Y) \supset g_*(\pi_1(Y)) \supset g^2_*(\pi(Y)) \supset \cdots
\] 
where $\pi_1(Y)$ is the fundamental group of $Y$ and $g_*$ is the map induced by $g$. Associated to this chain of finite index subgroups is an odometer action. This is an action of $\pi_1(Y)$ on a Cantor set, $\Omega$. Furthermore, the action is minimal in general and is free in our case (see Proposition \ref{actionFree}). In particular, the transformation groupoid associated to this action (denoted by $\Omega \rtimes \pi_1(Y)$) satisfies the hypotheses of the HK-conjecture and because the action is free is in addition principal. 

Next, using results of Scarparo \cite[Section 2.2]{Scarparo} and the Baum--Connes conjecture, the $K$-theory of the reduced $C^*$-algebra of $\Omega \rtimes \pi_1(Y)$ is shown to be the inductive limit group associated to an inductive system of the form:
\[ K_*(Y) \rightarrow K_*(Y) \rightarrow K_*(Y) \rightarrow \ldots  \]
where $K_*(Y)$ is the $K$-homology of $Y$. Likewise, using \cite[Section 2.3]{Scarparo}, the homology of $\Omega \rtimes \pi_1(Y)$ is shown to be the inductive limit group associated to an inductive system of the form:
\[ H_*(Y) \rightarrow H_*(Y) \rightarrow H_*(Y) \rightarrow \ldots  \]
where $H_*(Y)$ is the homology of $Y$. Key to both these inductive limit results is the fact that $Y$ is a model for $B\pi_1(Y)$. 

Based on the structure of these inductive limits, the problem is reduced to constructing a flat manifold where the $K$-homology and homology are not isomorphic (see Theorem \ref{torsionNotIso} below), while at the same time controlling the maps in the inductive limits. 

In regard to the first of these requirements, the following is proved in Section \ref{secExistence}: 
\begin{theorem} \label{torsionNotIso}
For any $d\ge 9$, there exists a $d$-dimensional flat manifold $Y$ with the property that
\[ |T(K_*(Y))| < |\bigoplus_{i} T(H_{*+2i}(Y))| \]
 where, for a finitely generated abelian group $G$, $T(G)$ denotes its torsion subgroup and $|T(G)|$ denotes the number of elements in $T(G)$ (which is finite in our situation).
 \end{theorem}
The construction of $Y$ satisfying Theorem \ref{torsionNotIso} relies on the theory of real Bott manifolds, see \cite{MR3273875, MR2576506} and references therein, and the Atiyah--Hirzebruch spectral sequence. Although not a direct application of \cite[Remark 6.12]{MR4030921}, our construction has a similar flavour. Then, using a result of Epstein and Shub \cite{EpSh} and the inductive limits discussed above, it is shown that for any flat manifold, there is an expanding endomorphism such that 
\[
 T(H_*(\Omega \rtimes \pi_1(Y))) \cong T(H_*(Y)).
\]
Combining this last equation with the fact that
\[ 
|T(K_*(C^*_r(\Omega \rtimes \pi_1(Y))))| \le |T(K_*(Y))|
\] 
and Theorem \ref{torsionNotIso} completes the construction of the counterexample. Finally, it is shown that every transformation groupoid associated to an odometer constructed from a flat manifold and expanding endomorphism (via the process discussed above) satisfies the rational version of the HK-conjecture, see Theorem \ref{mainResultQ} for the precise statement.

As reader might have noticed the construction involves quite a few ``moving parts". I would encourage the reader to familiarize themselves with \cite[Sections 2.1-2.3]{Scarparo} and \cite[Section 1]{ShubExp}. I have followed \cite[Sections 2 and 3]{MR3273875} when considering real Bott manifolds in Section \ref{secExistence}. In addition, some basic knowledge of flat manifolds and the Atiyah--Hirzebruch spectral sequence is also assumed, although I have explicitly listed the facts used.

In the next few paragraphs, future work is discussed. It is worth noting that the dimension of the flat manifold constructed in Section \ref{secExistence} is nine (or more) and that if a flat manifold satisfies the conclusion of Theorem \ref{torsionNotIso}, then its dimension must be greater than or equal to four. A systematic study of flat manifolds as in Theorem \ref{torsionNotIso} would be an interesting future project, especially in light of the positive results for low dimensional examples (see in particular, \cite[Corollary C]{Rufus} and \cite[Section 3]{PV1}). In particular, one can show that the dynamic asymptotic dimension of an odometer associated to a flat manifold and expanding endomorphism is the dimension of the manifold. As such, for any $d\ge 9$, we have a counterexample with dynamic asymptotic dimension equal to $d$. It would be interesting to obtain examples with smaller dynamic asymptotic dimension.

Based on the positive result of \cite[Corollary C]{Rufus}, one might ask if the following conjecture holds:
\begin{conjecture} \label{conjectureNew}
Suppose that $\mathcal{G}$ is a second countable, \'etale, principal, minimal, ample groupoid. Then there exists a (possibly different) groupoid $\tilde{\mathcal{G}}$ that is second countable, \'etale, principal, minimal, and ample such that 
\[ K_*(C^*_r(\mathcal{G})) \cong K_*(C^*_r(\tilde{\mathcal{G}})) \] 
and the HK-conjecture holds for $\tilde{\mathcal{G}}$.
\end{conjecture}
One approach to this conjecture would be to study the range of the $K$-theory of groupoids satisfying the HK-conjecture (e.g., by satisfying the hypotheses of \cite[Corollary C]{Rufus} or ideally generalizations of it). As stated Conjecture \ref{conjectureNew} would not be useful for computations. However, one could hope that there is an explicit construction of $\tilde{\mathcal{G}}$ from $\mathcal{G}$ that would facilitate computations.

Although this paper makes no reference to Smale spaces. I would like to mention that there is a connection between the unstable relation of a Smale space with totally disconnected stable sets and odometer actions, see \cite[page 194]{Put} for a specific case. In future work, this connection will be explored in detail. For now, it seems appropriated to mention that the counterexample in the present paper can be used to show that there is a counterexample to the HK-conjecture in the class of groupoids obtained from the unstable relation of Smale spaces with totally disconnected stable sets. This is of interest in light of recent results of Proietti and Yamashita \cite{PV2} connecting the homology of \'etale groupoids to Putnam's homology theory for Smale spaces \cite{PutHom}.

\section*{Acknowledgments}

I would like to thank Cl\'ement Dell'Aiera for bringing this problem to my attention at a talk he gave at the Great Plains Operator Theory Symposium (GPOTS) 2021. The organizers of GPOTS 2021 (Greg Knese, John McCarthy, Yanli Song, Xiang Tang, and Brett Wick) are also to be thanked for their organization of a wonderful conference during these these difficult times of lockdowns and quarantine. I would like to thank Rachel Chaiser, Ian Putnam, and Andrew Stocker for discussions surrounding an expanding endomorphism of the Klein bottle that was useful in developing the ideas of the present paper. While preparing this paper, I benefited from questions on MathOverFlow asked by Nicolas Boerger, Daniel Ramras, and Efton Park. The answer to Nicolas Boerger's question given by Benjamin Antieau was particularly useful. I thank Magnus Goffeng for various helpful discussions. Finally, I would like to the thank the referee for noticing an issue in the original formulation of Theorem \ref{inductiveLimitKtheory} and a number of other useful comments.

\section{Preliminaries} \label{secPrelim}
\subsection{Groupoids}
Let $\mathcal{G}$ be a groupoid. Its unit space is denoted by $\mathcal{G}^{(0)}$ and its range and source maps by $r, s : \mathcal{G} \to \mathcal{G}^{(0)}$. The ordered pair $g, h \in \mathcal{G}$ is composable if $s(g) = r(h)$ and their composition is denoted $gh$.  The inverse of $g \in \mathcal{G}$ is denoted $g^{-1}$. In this paper, all groupoids will be locally compact, Hausdorff, second countable, with compact unit space. Moreover, all groupoids in the paper will be \'etale, meaning that $r$ and $s$ are local homeomorphisms. In this case $\mathcal{G}^{(0)}$ is an open subset of $\mathcal{G}$ and the Haar system is given by counting measures. We say that $\mathcal{G}$ is principal if for each $x \in \mathcal{G}^{(0)}$ the isotropy group
\[ \mathcal{G}^x_x := \{ g \in \mathcal{G} \mid s(g) = r(g) = x \} \] 
is trivial (i.e., equal to $\{ x \}$). A groupoid, $\mathcal{G}$, is essentially principal if the interior of the set $\{ g \in \mathcal{G} \mid s(g)=r(g) \}$ is $\mathcal{G}^0$. Notice that principal implies essentially principal, but the converse is false. A groupoid is ample if its unit space is totally disconnected (e.g., the Cantor set).

To a groupoid satisfying the assumptions above, $\mathcal{G}$, one can associate its reduced groupoid $C^*$-algebras. The resulting $C^*$-algebra is denoted by $C^*_r(\mathcal{G})$. The computation of the $K$-theory of $C^*_r(\mathcal{G})$ is an important problem in $C^*$-algebra theory. 

The homology of $\mathcal{G}$ was defined in \cite{MR1752294} and will be denoted by $H_*(\mathcal{G})$. Some basic facts of this theory are the following:
\begin{enumerate}
\item If $X$ is a finite CW-complex with the trivial groupoid structure, then the groupoid homology is isomorphic to the standard cohomology of $X$. 
\item If $G$ is a group, then $H_*(G)$ is isomorphic to the group homology of $G$. Hence if $BG$ is the classifying space of $G$, then $H_*(G) \cong H_*(BG)$.
\end{enumerate}
In the present paper, only the second item and some results in \cite{Scarparo} relating groupoid homology to classical homology will be needed. As such, a detailed introduction to groupoid homology is not needed.

With this notation introduced, Matui's HK-conjecture \cite{MR3552533} for principal groupoids is the following:
\begin{conjecture}
Suppose that $\mathcal{G}$ is a second countable, \'etale, principal, minimal, ample groupoid. Then
\[ K_*(C^*_r(\mathcal{G})) \cong \bigoplus_{i} H_{*+2i}(\mathcal{G}) \] 
\end{conjecture}
The rational version of this conjecture (again in the principal case) is the following:
\begin{conjecture}
Suppose that $\mathcal{G}$ is a second countable, \'etale, principal, minimal, ample groupoid. Then
\[ K_*(C^*_r(\mathcal{G})) \otimes \Q \cong \bigoplus_{i} H_{*+2i}(\mathcal{G})\otimes \Q \] 
\end{conjecture}
We will provide a counterexample to the first of these conjectures, but our examples satisfying the weaker rational HK-conjecture. If one weakens the assumption of being principal in these two conjectures to being essentially principal (as in Matui's original formulation \cite{MR3552533}) then there are counterexamples to both, see \cite{Scarparo}. 

It is worth noting that all groupoids in the present paper are amenable.

\subsection{Group theory considerations}

Let $G$ be an abelian group. The torsion subgroup of $G$ is denoted by $T(G)$. When $G$ is finitely generated, $T(G)$ is finite. The number of elements in a finite group, $F$, is denoted by $|F|$. Given an inductive systems of groups of the form,
\[ G \xrightarrow{\beta_0} G \xrightarrow{\beta_1} G \xrightarrow{\beta_2} \cdots  \]
the associated inductive limit group is denoted by $\lim_{\rightarrow} (G,  \beta_i)$. An element is denoted by $[ \gamma, k ]$ where $\gamma \in G$ and $k\in \N$, see \cite{PutHom} for more on inductive limits and this notation.

The next few results are certainly known, but it is useful to have them recorded for later use. 
\begin{proposition}
Suppose $G$ is a finitely generated abelian group, $n$ is an integer greater than or equal to one, and $\alpha: G\rightarrow G$ and $\beta: G \rightarrow G$ are group homomorphisms satisfying $\alpha \circ \beta = $ multiplication by $n$. If for each $\gamma \in T(G)$, $n \gamma =\gamma$ then $\beta|_{T(G)} : T(G) \rightarrow T(G)$ is an isomorphism.
\end{proposition}
\begin{proof}
By assumption, for any $\gamma\in T(G)$, $(\alpha \circ \beta)(\gamma)=n \gamma = \gamma$. Hence, $\beta|_{T(G)}$ is injective. But $T(G)$ is finite (since $G$ is finitely generated abelian) so $\beta|_{T(G)}$ is also surjective.
\end{proof}

\begin{proposition} \label{inductiveLimitGenProp}
Suppose $G$ is a finitely generated abelian group, $n$ is an integer greater than or equal to one, and for each $i$, $\alpha_i: G\rightarrow G$ and $\beta_i: G \rightarrow G$ are group homomorphisms satisfying $\alpha_i \circ \beta_i = $ multiplication by $n$. If for each $\gamma \in T(G)$, $n \gamma =\gamma$, then the map $\Phi : T(G) \rightarrow T(\lim_{\rightarrow} (G, \beta_i))$ defined via
\[ \gamma \mapsto [\gamma, 0] \]
is an isomorphism.
\end{proposition}
\begin{proof}
It is clear that $\Phi$ is a group homomorphism. That $\Phi$ is injective follows because each $\beta_i$ is injective. 

To show that $\Phi$ is onto, let $[\tilde{\gamma}, k] \in T(\lim_{\rightarrow} (G, \beta))$. By the definition of the inductive limit group, 
\[ k_1 (\beta_{k_2} \circ \beta_{k_2-1} \circ \cdots \circ \beta_{k+1})(\tilde{\gamma})=0 \] 
for some $k_1, k_2 \in \N$. Applying $\alpha_{k+1} \circ \cdots \alpha_{k_2-1} \circ \alpha_{k_2}$ leads to
\[ k_1 \cdot n^{k_2}(\tilde{\gamma})=0. \]
It follows that $\tilde{\gamma}$ is in $T(G)$. 

The previous proposition ensures that for each $i$, $(\beta_i)|_{T(G)}$ is an isomorphism. Hence, we can form 
\[ \gamma=(((\beta_0)|_{T(G)})^{-1} \circ \cdots \circ ((\beta_k)|_{T(G)})^{-1})(\tilde{\gamma}). \] 
One checks that 
\[ \Phi(\gamma)=[ \gamma , 0]=[ (\beta_{k} \circ \cdots \beta_0)({\gamma}), k]=[ \tilde{\gamma}, k ] \]
as required.
\end{proof}

\begin{proposition} \label{inductiveLimitGenPropTorsionIncluded}
Suppose $G$ is a finitely generated abelian group, 
\[ G \xrightarrow{\beta_0} G \xrightarrow{\beta_1} G \xrightarrow{\beta_2} \cdots  \]
is an inductive system, and $\lim_{\rightarrow}(G, \beta_i)$ is the inductive limit group. Then $T(\lim_{\rightarrow}(G, \beta_i))$ is a finite group and $|T((G, \beta_i))|\le |T(G)|$.
\end{proposition}
\begin{proof}
Since $G$ is finitely generated abelian, $T(G)$ is a finite group. Take $M=|T(G)|+1$ elements in $T(\lim_{\rightarrow}(G, \beta_i))$, which we write as
\[ [\gamma_1, k_1], [\gamma_2, k_2], \ldots, [\gamma_M, k_M] \]
As in the proof of the previous result, we can assume that $\gamma_i \in T(G)$ for each $i=1, \ldots, M$. Furthermore, by applying the connecting maps in the inductive system, we can assume that $k_1=k_2= \ldots =k_M$. It follows from the pigeonhole principle that there exists $i \neq j$ such that $\gamma_i = \gamma_j$, which completes the proof.
\end{proof}

\subsection{Flat manifolds} \label{secFlatMan}
A flat manifold refers to a closed, connected, Riemannian flat manifold. Throughout this section and the rest of the paper, $Y$ is a flat manifold with dimension $d$. Examples of flat manifolds in low dimensions include the circle, the torus and the Klein bottle, see \cite{MR862114} for more details and many more examples (see in particular page 41 of \cite{MR862114}). The following basic properties will be used. The first two can be found in for example \cite{MR862114} and the third follows from the first two.
\begin{enumerate}
\item The fundamental group of $Y$, $\pi_1(Y)$, is torsion-free and fits within the following short exact sequence:
\[ 0 \rightarrow \Z^d \rightarrow \pi_1(Y) \rightarrow F \rightarrow 0 \]
where $\Z^d$ is maximal abelian and $F$ is a finite group (called the holonomy).
\item The short exact sequence in the previous statement is obtained from a $|F|$-fold cover of $Y$ by the $d$-torus. This covering map is denoted by $p$.
\item It follows from the previous statements that $Y$ is a model for the classifying space $B(\pi_1(Y))$ and likewise $\R^d$ is a model for $E(\pi_1(Y))$. Furthermore, $\pi_1(Y)$ is amenable and hence the Baum--Connes conjecture with coefficients holds for $\pi_1(Y)$.
\end{enumerate}
The next result is well-known, see for example \cite[Lemma 2.7]{MR2581917}.
\begin{proposition} \label{flatManTorsionHomology}
If $x \in T(H_*(Y))$, then the order of $x$ divides $|F|$. In particular, for any $k \in \N$ and $x \in T(H_*(Y))$, $(|F|+1)^k x =x$.
\end{proposition}

\subsection{The Atiyah--Hirzebruch spectral sequence} \label{SecAHSS} A number results concerning the Atiyah--Hirzebruch spectral sequence are collected for future use. None of them are new and all are likely well-known to experts. The results are summarized here so that the computations in Section \ref{secExistence} are as easy as possible. Although a number of the results below hold in more generality, throughout $X$ is assumed to be a closed connected orientable manifold with dimension $d$.

Before getting to the spectral sequence, a few fundamental properties of the Steenrod square maps (see for example \cite[Chapter 10 Section 8]{MR2377868}) are discussed as they are relevant for the differentials in the spectral sequences. Recall that for each non-negative integer $m$, the Steenrod square map of degree $m$ is a map ${\rm Sq}^m : H^k(X;\Z/2\Z) \rightarrow H^{k+m}(X; \Z/2\Z)$. We will only need these maps when $m=2$ or $3$ and the formal definition is not required. The only properties needed are the following:
\begin{enumerate}
\item for $k=m$ it maps $x$ to $x \cup x$ (we denote $x\cup x$ by $x^2$);
\item we have that ${\rm Sq}^3= r \circ \beta \circ {\rm Sq}^2$ where 
\begin{enumerate}
\item $\beta : H^{k+2}(X; \Z/2\Z) \rightarrow H^{k+3}(X)$ is the Bockstein map and 
\item $r: H^{k+3}(X) \rightarrow H^{k+3}(X;\Z/2\Z)$ is the reduction mod two map.
\end{enumerate}
\end{enumerate}

We now move to the spectral sequences. The reader is invited to review \cite[Chapter 21]{MR2377868} for the notation used here. In addition, note that $p$ and $q$ have been suppressed from the notation of the differentials. We have the following fundamental properties of the Atiyah--Hirzebruch spectral sequences for $K$-theory $\{ E_m^{p, q} \}$ and $K$-homology $\{ E^m_{p,q} \}$ (recall that $X$ is an orientable manifold):
\begin{enumerate}
\item $E^3_{p, q} \cong E^2_{p, q} \cong \left\{ \begin{array}{cc} H_p(X) & q \hbox{ is even} \\ 0 & q \hbox{ is odd} \end{array} \right. $
\item $E_3^{p, q} \cong E_2^{p, q} \cong  \left\{ \begin{array}{cc} H^p(X) & q \hbox{ is even} \\ 0 & q \hbox{ is odd} \end{array} \right. $
\item The Atiyah--Hirzebruch spectral sequence for $K$-homology is a module over the Atiyah--Hirzebruch spectral sequence for $K$-theory. In particular, if $x\in H^*(X)$ and $[X]$ is the fundamental class of $X$ associated to a particular orientation, then we have
\[ d^3(x \cap [X])= d_3(x) \cap [X] \pm x \cap d^3([X]).
\]
\item The differential $d_3 : H^k(X) \rightarrow H^{k+3}(X)$ is given by $\beta \circ {\rm Sq}^2 \circ r$ where
\begin{enumerate}
\item $r: H^k(X) \rightarrow H^k(X;\Z/2\Z)$ is the reduction mod two map,
\item ${\rm Sq}^2: H^k(X; \Z/2\Z) \rightarrow H^{k+2}(X; \Z/2\Z)$ is the Steenrod square map, and
\item $\beta : H^{k+2}(X; \Z/2\Z) \rightarrow H^{k+3}(X)$ is the Bockstein map.
\end{enumerate}
\item If $d^3 \neq 0$, then 
\[ |T(K_*(X))| < |\bigoplus_{i} T(H_{*+2i}(X))|. \]
A short justification of this fact proceeds as follows. Since the Chern character is an isomorphism after tensoring with the rational numbers, the differentials in the Atiyah--Hirzebruch spectral sequence are pure torsion morphisms (see for example \cite[Chapter 21: Remark 4.7 and Theorem 4.8]{MR2377868} in the case of $K$-theory rather than $K$-homology). This implies that for all $p$, $q$
\[
|T(E^4_{p,q})| \le |T(E^3_{p, q})|= \left\{ \begin{array}{cc} |T(H_p(X))| &  q \hbox{ is even} \\ 1 & q \hbox{ is odd} \end{array} \right. 
\] 
Moreover, since $d^3\neq 0$, there exists $p$, $q$ such that
\[
|T(E^4_{p,q})| < |T(E^3_{p, q})|= \left\{ \begin{array}{cc} |T(H_p(X))| &  q \hbox{ is even} \\ 1 & q \hbox{ is odd} \end{array} \right. 
\]
where we have used the fact that the relevant torsion groups are finite because all groups considered are finitely generated abelian. Likewise,
for all $p$, $q$
\[
|T(E^{\infty}_{p,q})| \le |T(E^3_{p, q})|= \left\{ \begin{array}{cc} |T(H_p(X))| &  q \hbox{ is even} \\ 1 & q \hbox{ is odd} \end{array} \right. 
\] 
and for at least one $p$, $q$
\[
|T(E^{\infty}_{p,q})| < |T(E^3_{p, q})|= \left\{ \begin{array}{cc} |T(H_p(X))| &  q \hbox{ is even} \\ 1 & q \hbox{ is odd} \end{array} \right. 
\]
Using this and the fact that the Atiyah--Hirzebruch spectral sequence converges to the $K$-homology of $X$, it follows that (see for example \cite[Chapter 21: Assertion 4.3 and 4.5]{MR2377868} in the context of $K$-theory)
\[ |T(K_*(X))|  < |T(\bigoplus_{i} H_{*+2i}(X))|. \]
This completes the proof.
\end{enumerate} 

Finally, recall that the Stiefel--Whitney classes of $X$ (see for example \cite[Chapter 10 Definition 3.7]{MR2377868}) are classes $w_i(X) \in H^i(X; \Z/2\Z)$ and the integral Stiefel--Whitney classes of $X$ are classes $W_i(X)\in H^i(X)$. We only need the following property:
\[ r(W_3(X))=w_3(X) \]
where $r: H^3(X) \rightarrow H^3(X;\Z/2\Z)$ is the reduction mod two map.
\begin{proposition}
Suppose $X$ is an orientable manifold and $d_3$ is non-zero. Then $d^3$ is also non-zero.
\end{proposition}
\begin{proof}
Fix an orientation on $X$ to obtain the fundamental class $[X] \in H_d(X)$. If $d^3([M]) \neq 0$ then we are done. 

Otherwise, take $x\in H^*(X)$ such that $d_3(x) \neq 0$. Then, using the module structure, Poincar\'e duality, and the fact that $d^3([X])=0$, we obtain
\[
d^3( x \cap [X]) = d_3(x) \cap [X] \neq 0.
\]
\end{proof}
\begin{proposition} \label{propAboutNonEqualTorsion}
If $X$ is an orientable manifold and $w_3(X)^2$ is non-zero, then $d_3$ is nonzero. In particular, under these assumptions on $X$,
\[ |T(K_*(X))| < |\bigoplus_{i} T(H_{*+2i}(X))| \]
\end{proposition}
\begin{proof}
Firstly, by assumption, ${\rm Sq}^3(w_3(X))=w_3(X)^2 \neq 0$. Then, since 
\[ {\rm Sq}^3(w_3(X))=(r \circ \beta \circ {\rm Sq}^2)(w_3(X)), \]
it follows that 
\[ (\beta \circ {\rm Sq}^2)(w_3(X))\neq 0. \]
Finally, since $r(W_3(X))=w_3(X)$, 
\[ d_3(W_3(X))=(\beta \circ {\rm Sq}^2 \circ r)(W_3(X))=(\beta \circ {\rm Sq}^2)(w_3(X))\neq 0. \]
This completes the proof of the first statement.

The ``in particular" part of the theorem follows using the previous proposition and Item (5) in the list of properties of the Atiyah--Hirzebruch spectral sequence.
\end{proof}
\section{Expansive endomorphisms of flat manifolds and odometers}
Throughout this section, $Y$ is a flat manifold and $g: Y \rightarrow Y$ is an expanding endomorphism. That is (see page 176 of \cite{ShubExp}) there exists $C>0$ and $\lambda>1$ such that $|| Tg^k v || \ge c\lambda^k ||v ||$ for each $v \in TY$ and strictly positive integer $k$. Here $|| \cdot ||$ denotes a fixed Riemannian metric, but it is worth noting that being expanding is independent of the choice of metric (although the particular constants $C$ and $\lambda$ do depend on the metric).

By \cite[Proposition 3]{ShubExp}, $g$ is a covering map and since $Y$ is compact, $g$ is a $n$-fold cover for some $n \ge 2$. By \cite[Theorem 1]{ShubExp}, $g$ has a fixed point $y_0$. We will use this as our based point, so $\pi_1(Y)$ denotes $\pi(Y, y_0)$. Associated to $g$ is a chain of finite index, proper subgroup inclusions:
\[
\pi_1(Y) \supset g_*(\pi_1(Y)) \supset g^2_*(\pi(Y)) \supset \cdots
\] 
The associated odometer is obtained as follows. Let
\[ \Omega = \lim_{\leftarrow} (\Omega_i, f_{i-1}^i) \]
where $\Omega_i= \pi_1(Y)/ g^i_*(\pi(Y))$ and $f_{i-1}^i$ is given by inclusion of cosets. Each $\Omega_i$ is a finite set (containing more than one element) and hence $\Omega$ is a Cantor set. An element in $\Omega$ can be written as
\[ (\gamma_0 \pi_1(Y), \gamma_1 g_*(\pi_1(Y)), \gamma_2 g^2_*(\pi_1(Y)), \ldots ). \]
The odometer action of $\pi_1(Y)$ on $\Omega$ is defined via
\[
\gamma \cdot (\gamma_0 \pi_1(Y), \gamma_1 g_*(\pi_1(Y)), \gamma_2 g^2_*(\pi_1(Y)), \ldots ) =(\gamma \gamma_0 \pi_1(Y), \gamma \gamma_1 g_*(\pi_1(Y)), \gamma \gamma_2 g^2_*(\pi_1(Y)), \ldots ).
\]
The odometer action is minimal, see for example \cite[Section 2.1]{MR2427048}. By Proposition 4 on page 181 of \cite{ShubExp}, 
\[ \bigcap_{k\ge 0} g^k_*(\pi_1(Y))=\{ e\} \]
but (again see \cite[Section 2.1]{MR2427048}) this is not enough to conclude that the action is free. This is because the subgroup $g_*(\pi_1(Y)) \subseteq \pi_1(Y)$ is typically not normal, see \cite[Corollary 1.18]{MR3784526}. Nevertheless the odometer action associated to an expanding endomorphism is indeed free and this is likely known. I was unable to find a precise reference so a proof has been included.
\begin{proposition} \label{actionFree}
The odometer action associated to an expanding endomorphism $g: Y \rightarrow Y$ is free.
\end{proposition}
\begin{proof}
To begin, fix a Riemannian metric on $Y$ and recall that $y_0$ is a fixed point of $g$. The odometer action can be describe in terms of preimages of $y_0$ with respect to $g$, $g^2$, etc. To do so, notice that there is a one-to-one correspondence between $g^{-1}(y_0)$ and cosets associated to the subgroup $g_*(\pi_1(Y))$. 

This correspondence is given as follows. Given a coset take a loop, $\gamma$, based at $y_0$ representing a class in that coset. Let $\tilde{\gamma}: [0,1] \rightarrow Y$ denote the unique lift of $\gamma$ to a path starting at $y_0$. Then $\tilde{\gamma}(1) \in g^{-1}(y_0)$. Furthermore, this defines the required one-to-one correspondence. One must check the process is well-defined and one-to-one but this follow from elementary properties of covering space theory.

Repeating this process with $g^{-2}(y_0)$, $g^{-3}(y_0)$, etc, we have that $\Omega$ is homeomorphic to 
\[
\{ (y_0, y_1, y_2, \ldots ) \mid y_0 \hbox{ is the fixed point above and }g(y_{i+1})=y_i \} 
\]
where the topology is the subspace topology. Furthermore, the odometer action with respect to this realization is given as follows. Let $\gamma$ be a loop based at $y_0$ representing an element in $\pi_1(Y)$ and $y_1 \in g^{-1}(y_0)$. Let $\tilde{\gamma}_1 : [0,1] \rightarrow Y$ be the unique lift of $\gamma$ to a path starting at $y_1$. Then $[\gamma] \cdot y_1:= \tilde{\gamma}_1(1)$. As with the discussion at the topological space level, by repeating this process one obtains the odometer action on the space  
\[
\{ (y_0, y_1, y_2, \ldots ) \mid y_0 \hbox{ is the fixed point above and }g(y_{i+1})=y_i \}.
\]
We can now show that the action is free. Let $\gamma$ be a loop based at $y_0$ representing a class in $\pi_1(Y)$, which we can and will assume is smooth so that it has a well-defined arclength. Suppose that 
\[
[\gamma] \cdot (y_0, y_1, y_2, \ldots )= (y_0, y_1, y_2, \ldots ).
\]
We must show that $[\gamma]$ is the identity in $\pi_1(Y)$. By the definition of the action discussed above, we have that the unique lift of $\gamma$ to a path, $\tilde{\gamma}_1:[0,1] \rightarrow Y$ starting at $y_1$ satisfies the following:
\begin{enumerate}
\item Because $[\gamma] \cdot y_1= y_1$, $\tilde{\gamma}_1$ is a loop based at $y_1$ (rather than just a path starting at $y_1$);
\item Because $g$ is expanding and $g \circ \tilde{\gamma}_1= \gamma$ (by the definition of lift), the arclength of $\tilde{\gamma}_1$ is less than or equal to $\frac{L}{C\lambda}$ where $L$ is the arclength of $\gamma$ and $C$, $\lambda$ are constants from the definition of expanding endomorphism.
\end{enumerate}
Noticing that $\tilde{\gamma}_1$ is a loop, we denote it by $\gamma_1$. The process applied (to $\gamma$) above can be applied to $\gamma_1$. We obtain a loop based at $y_2$, $\gamma_2$ with arclength less than or equal to $\frac{L}{C\lambda^2}$ that is a lift of $\gamma_1$. Continuing the process, for each positive integer $k$, we obtain a loop $\gamma_k$ based at $y_k$ with arclength less than or equal to $\frac{L}{C\lambda^k}$ that is a lift of $\gamma_{k-1}$. Since $\lambda>1$, there exists $k$ and open set $U\subseteq Y$ containing $y_k$ such that $U \cong \R^d$ and $\gamma_k(t)\in U$ for all $t\in [0,1]$. It follows that $\gamma_k$ is nullhomotopic. But then $\gamma= g^k \circ \gamma_k$ is also nullhomotopic and hence $[ \gamma ]$ is the identity in $\pi_1(Y)$.
\end{proof}
In summary, the properties discussed above imply the following.
\begin{proposition} \label{HKassumption}
Suppose (as above) $Y$ is a flat manifold and $g: Y \rightarrow Y$ is an expanding endomorphism. Then the transformation groupoid associated to the odometer action of $\pi_1(Y)$ on $\Omega$ is a second countable, \'etale, principal, minimal groupoid. Moreover, its unit space is the Cantor set.
\end{proposition}
\begin{proof}
The transformation groupoid of the action of a discrete group is always \'etale and the unit space in our situation is $\Omega$ (the Cantor set). The groupoid is clearly second countable. It is minimal because the action is minimal and it is principal because the action is free. 
\end{proof}

\section{The $K$-theory and homology of odometers}

\subsection{Homology}
As in the previous section, $Y$ is a flat manifold of dimension $d$, $g: Y \rightarrow Y$ is an expanding endomorphism (it is an $n$-fold cover), $\Omega$ is the Cantor set associated with the odometer action of $\pi_1(Y)$, and $\Omega \rtimes \pi_1(Y)$ is the associated transformation groupoid. 
\begin{theorem} \label{inductiveLimitHomology}
The homology of $\Omega \rtimes \pi_1(Y)$ is isomorphic to the inductive limit group: $\lim_{\rightarrow} (H_*(Y), \tilde{g}_i)$
where each $\tilde{g}_i : H_*(Y) \rightarrow H_*(Y)$ with the property that there exists $\tilde{h}_i: H_*(Y) \rightarrow H_*(Y)$ such that $\tilde{h}_i \circ \tilde{g}_i =$ multiplication by $n$. In particular, $\tilde{g}_i$ is a rational isomorphism.
\end{theorem}
\begin{proof}
By \cite[Proposition 2.4]{Scarparo}, the homology of $\Omega \rtimes \pi_1(Y)$ is isomorphic to the inductive limit group:
\[
\lim_{\rightarrow} (H_*(p^i_*(\pi_1(Y))), {\rm tr}_i^{i+1})
\]
where ${\rm tr}_i^{i+1}$ is the transfer map in group homology. By \cite[Proposition III.9.5 (ii)]{Brown}, the connecting maps have the required property. Moreover, $Y$ is a model for $B(\pi_1(Y))$ and for each $i$, $p^i_*(\pi_1(Y)) \cong \pi_1(Y)$. Hence (for each $i$) $H_*(p^i_*(\pi_1(Y))) \cong H_*(Y)$.

Finally, the ``in particular" part of the theorem follows from the fact that $\tilde{h}_i \circ \tilde{g}_i =$ multiplication by $n$.
\end{proof}

\subsection{$K$-theory}
As in the previous section, $Y$ is a flat manifold of dimension $d$, $g: Y \rightarrow Y$ is an expanding endomorphism (it is an $n$-fold cover), and $\Omega$ is the Cantor set associated with the odometer action of $\pi_1(Y)$. The (reduced) transformation groupoid $C^*$-algebra of the odometer actions is $C^*_r(\Omega \rtimes \pi_1(Y)) \cong C(\Omega) \rtimes_r \pi_1(Y)$ where we note that $\pi_1(Y)$ is amenable so there is no difference between the full and reduced $C^*$-algebras so we will drop the $r$ from the notation. 
\begin{theorem} \label{inductiveLimitKtheory}
The $K$-theory of $C(\Omega) \rtimes \pi_1(Y)$ is isomorphic to the inductive limit group:
\[ \lim_{\rightarrow} ( K_*(Y), \hat{g}_i) \]
Moreover, each map $\hat{g}_i : K_*(Y) \rightarrow K_*(Y)$ is a rational isomorphism.
\end{theorem}
\begin{proof}
For the inductive limit part of the proof, we begin with the fact (see \cite{Scarparo} page 2544) that 
\[ C(\Omega) \rtimes \pi_1(Y) \cong \lim_{\rightarrow} C(\Omega_i) \rtimes \pi_1(Y) \]
where $\Omega_i=\pi_1(Y)/ g^i_*(\pi_1(Y))$ and the map in the inductive limit is obtained from the map $\Omega_{i+1} \rightarrow \Omega_i$ defined using $g^{i+1}_*(\pi_1(Y)) \subseteq g^i_*(\pi_1(Y))$. Furthermore, \cite[Proposition 2.3]{Scarparo}, implies that, for each $i$, 
\[ C(\Omega_i) \rtimes \pi_1(Y) \cong M_{n^i}(\C) \otimes C^*_r(\pi_1(Y)) \] 
where we have used the fact that $g$ is an $n$-fold cover and (for each $i$) $p^i_*(\pi_1(Y)) \cong \pi_1(Y)$. We have that
\[
K_*(C(\Omega_i) \rtimes \pi_1(Y)) \cong K_*(M_{n^i}(\C) \otimes C^*_r(\pi_1(Y))) \cong K_*(C^*_r(\pi_1(Y))) \cong K_*(Y)
\]
where in the last step we have used the fact that $\pi_1(Y)$ satisfies the Baum--Connes conjecture, $\pi_1(Y)$ is torsion-free, and $Y$ is a model for $B(\pi_1(Y))$.

For the second part of the proof, again by the Baum--Connes conjecture (now with coefficients), for each $i$, 
\[
K_*(C(\Omega_i) \rtimes \pi_1(Y)) \cong KK^{\pi_1(Y)}_*(C_0(\R^d), C(\Omega_i))
\]
and the connecting maps in the inductive limit are given by 
\[ (g_i)_* : KK^{\pi_1(Y)}_*(C_0(\R^d), C(\Omega_i)) \rightarrow KK^{\pi_1(Y)}_*(C_0(\R^d), C(\Omega_{i+1})) \]
where at the space level $g_i: \Omega_{i+1} \rightarrow \Omega_i$ is defined using $g^{i+1}(\pi_1(Y)) \subseteq g^i_*(\pi_1(Y))$. The map $g_i$ is a covering map and hence there is a transfer map 
\[ (g_i)! : KK^{\pi_1(Y)}_*(C_0(\R^d), C(\Omega_{i+1})) \rightarrow KK^{\pi_1(Y)}_*(C_0(\R^d), C(\Omega_i)). \]
One can then show that $(g_i)_*$ is a rational isomorphism directly (compare with the proof of \cite[Lemma 4.2 Part 2]{MR3073917} in the context of $K$-theory) or using the Chern character to relate the inductive limit in the present theorem with the one for homology in the previous section; the details are omitted. 
\end{proof}

\begin{remark}
With a bit more work, one can show that the inductive limits in both Theorems \ref{inductiveLimitHomology} and \ref{inductiveLimitKtheory} are stationary. However this is not needed for the results of the present paper. The fact that these inductive limits are stationary is similar to my previous work with Allan Yashinski in \cite{MR4138909} concerning the stable groupoid $C^*$-algebra of a Smale space with totally disconnected stable sets.
\end{remark}

\subsection{Main results}
\begin{theorem} \label{mainResultTorsion}
Suppose that $Y$ is a flat manifold. Then there exists an expanding endomorphism $g: Y \rightarrow Y$ such that 
\[
T(H_*(\Omega \rtimes \pi_1(Y))) \cong T(H_*(Y))
\]
\end{theorem}
\begin{proof}
By the main result of \cite{EpSh} (see the theorem on page 140 of \cite{EpSh}), there exists an expanding endomorphism $g: Y\rightarrow Y$ satisfying
\[ \begin{CD}
		(S^1)^d @>p >> Y \\
		@V\times m VV @Vg VV \\
		(S^1)^d @>p >> Y \\
\end{CD} \]
where 
\begin{enumerate}
\item $p$ is the cover of $Y$ by the torus discussed in Section \ref{secFlatMan} and
\item $\times m$ is the map multiplication by $m$ with $m=|F|+1$ ($F$ was also discussed in Section \ref{secFlatMan}). 
\end{enumerate}

It follows that $g$ is an $n$-fold cover with $n=m^d=(|F|+1)^d$. By Proposition \ref{flatManTorsionHomology} and the fact that the homology of $Y$ are finitely generated, we can apply Proposition \ref{inductiveLimitGenProp}. The result then follows from an application of Proposition \ref{inductiveLimitGenProp} to the inductive limit in Theorem \ref{inductiveLimitHomology}.
\end{proof}

\begin{theorem} \label{mainResultQ}
Suppose $Y$ is a flat manifold and $g: Y\rightarrow Y$ is an expanding endomorphism. Then 
\[
K_*(C(\Omega) \rtimes \pi_1(Y)) \otimes \Q \cong K_*(Y)\otimes \Q \hbox{ and }H_*(\Omega \rtimes \pi_1(Y)) \otimes \Q \cong H_*(Y) \otimes \Q.
\]
In particular, the rational HK-conjecture holds for $\Omega \rtimes \pi_1(Y)$.
\end{theorem}
\begin{proof}
Since taking the tensor product with the rationals respects inductive limits, we can apply Theorem \ref{inductiveLimitKtheory} to obtain 
\[ K_*(C(\Omega) \rtimes \pi_1(Y)) \otimes \Q  \cong  \lim_{\rightarrow} ( K_*(Y)\otimes \Q , \hat{g}_i\otimes id_{\Q} ). \]
Furthermore, Theorem \ref{inductiveLimitKtheory} implies that $\hat{g}_i\otimes id_{\Q}$ is invertible. This completes the proof for $K$-theory. For homology, the proof is the same with the use of Theorem \ref{inductiveLimitKtheory} replaced by Theorem \ref{inductiveLimitHomology}.

Finally, the rational HK-conjecture holds because the Chern character (from the $K$-homology of $Y$ to the even/odd homology of $Y$) is a rational isomorphism.
\end{proof}

Based on Theorem \ref{mainResultTorsion} we have the following:
\begin{corollary} \label{mainCor} Suppose $Y$ is a flat manifold with 
\[ |T(K_*(Y))| < |\bigoplus_{i} T(H_{*+2i}(Y))|. \]
Then there exists an expanding endomorphism $g: Y\rightarrow Y$ such that the transformation groupoid associated to the odometer action of $\pi_1(Y)$ is a counterexample to the HK-conjecture. Moreover, the relevant groupoid is principal.
\end{corollary}
\begin{proof}
Take $g: Y \rightarrow Y$ as in Theorem \ref{mainResultTorsion}. By Proposition \ref{HKassumption}, the groupoid $\Omega \rtimes \pi_1(Y)$ satisfies the hypotheses of the HK-conjecture and is principal. Using Proposition \ref{inductiveLimitGenPropTorsionIncluded} and Theorem \ref{mainResultTorsion}, we have 
\[
|T(K_*(C^*(\Omega \rtimes \pi_1(Y)))| \le |T(K_*(Y))| < |T( \bigoplus_{i} (H_{*+2i}(Y)))| =  \left| T\left(\bigoplus_{i} H_{*+2i}(\Omega \rtimes \pi_1(Y))\right) \right|.
\]
In particular, $K_*(C^*(\Omega \rtimes \pi_1(Y))) \not\cong \bigoplus_{i} H_{*+2i}(\Omega \rtimes \pi_1(Y))$.
\end{proof}
The goal of the next section is the construction of a flat manifold satisfying the condition in the previous corollary. It is worth noting that if $Y$ satisfies
\[ T(K_*(Y)) \not\cong \bigoplus_{i} T(H_{*+2i}(Y)), \]
then ${\rm dim}(Y) \geq 4$, see \cite[Proposition 2.1 (ii)]{MR1951251}.

\section{The existence of the required flat manifold} \label{secExistence}
Our goal is the construction of a flat manifold satisfying the condition in Corollary \ref{mainCor}. In fact, the following will be proved:
\begin{theorem} \label{existDimNine}
For each $d\ge 9$, there exists a flat manifold $Y$ of dimension $d$ with the property that
\[ |T(K_*(Y))|< |\bigoplus_{i} T(H_{*+2i}(Y))|. \]
 Recall that for an abelian group $G$, $T(G)$ denotes its torsion subgroup.
 \end{theorem}
Based on Proposition \ref{propAboutNonEqualTorsion} in Section \ref{SecAHSS}, (for each $d\ge 9$) we need only construct a flat manifold $Y$ (of dimension $d$) such that
\[ w_1(Y)=0 \hbox{ and }w_3(Y)^2 \neq 0
\]
Notice that $w_1(Y)=0$ implies that the $Y$ is orientable, which was a standing assumption in Section \ref{SecAHSS}.

I would like to recommend the reader review Section 2 and 3 from \cite{MR3273875} for an introduction to the important class of flat manifolds called real Bott manifolds. 

The case of $d=9$ is considered first. Using the notation of \cite{MR3273875} (see in particular, page 1017), let $Y(A)$ be the real Bott manifold associated to the matrix
\[ A = \left[ \begin{array}{ccccccccc}0 & 1 & 0 & 0 & 0 & 0 & 0 & 0 & 1 \\ 0 & 0 & 1 & 0 & 0 & 0 & 0 & 0  & 1 \\ 0 & 0 & 0 & 1 & 0 & 0 & 0 & 0 & 1 \\ 0 & 0 & 0 & 0 & 1 & 0 & 0 & 0 & 1 \\ 0 & 0 & 0 & 0 & 0 & 1 & 0 & 0 & 1 \\ 0 & 0 & 0 & 0 & 0 & 0 & 1 & 0 & 1 \\ 0 & 0 & 0 & 0 & 0 & 0 & 0 & 1 & 1  \\ 0 & 0 & 0 & 0 & 0 & 0 & 0 & 0 & 0  \\ 0 & 0 & 0 & 0 & 0 & 0 & 0 & 0 & 0  \end{array} \right] \]
It is worth noting that the dimension of $Y(A)$ is in fact $d=9$. The cohomology of $Y(A)$ with coefficients in $\Z/2\Z$ is determined by $A$. This was proved in \cite[Lemma 2.1]{MR2576506} and can also be found on page 1020 of \cite{MR3273875}. We have that  
\[
H^*(Y(A); \Z/2\Z) \cong (\Z/2\Z)[x_1, \ldots, x_d] / ( x_j^2=x_j \sum_{i=1}^dA_{ij}x_i \mid j=1, \ldots , d)
\]
For our specific choice of $A$, the relations are as follows:
\begin{equation} \label{relationX}
x_1^2=0, x_2^2=x_2x_1, \ldots  , x_8^2=x_8x_7, \hbox{ and } x_{9}^2=x_{9}(x_7+ \ldots +x_1)
\end{equation}
Moreover (again see page 1020 of \cite{MR3273875}) the classes $w_1(Y(A))$ and $w_3(Y(A))$ are given respectively by
\[
w_1(Y(A))=\sum_{i=1}^d y_i
\]
and
\begin{equation} \label{w3formula}
w_3(Y(A))=\sum_{1\le i < j < k\le d} y_iy_jy_k
\end{equation}
where, in general (see page 1017 of \cite{MR3273875})
\[ y_i = \sum_{k=1}^{i-1} A_{k, i} x_k.
\]
In our specific situation
\begin{equation} \label{relationY}
y_1=0, y_2=x_1, \ldots, y_8=x_7 \hbox{ and }y_{9}=x_1+ \ldots +x_7
\end{equation}

By either applying \cite[Lemma 2.2]{MR2576506} or direct computation, one checks that $w_1(Y(A))=0$ and hence $Y(A)$ is orientable.

To show that $w_3(Y(A))^2$ is non-zero is more involved. We will show that there is an odd number of terms of the form $x_1 x_2 x_4 x_5 x_6 x_7$ in the expression of $w_3(Y(A))^2$. To begin, we consider terms of the form $y_i y_j y_k$ where $1\le i < j < 8$. Collecting terms and using Equation (\ref{relationY}), we have
\[ y_i y_j ( y_{j+1} + \ldots y_9) = x_{i-1}x_{j-1}(x_{j+1}+\ldots x_{7} + x_1+\ldots +x_7)=x_{i-1}x_{j-1}(x_1 +\ldots + x_{j-1}) \]
where we have used the fact that $2 x_l=0$ since we are working in $H^*(Y(A); \Z/2\Z)$. Importantly for us, none of these terms contain an $x_7$.

This leaves terms of the form $y_i y_8 y_9$ where $i=2, \ldots , 7$ (where we have used the fact that $y_1=0$). Using Equation (\ref{relationY}) and the fact $x_7^2=x_7x_6$, we have that
\[ y_i y_8 y_9= x_{i-1}x_7(x_1 +x_2 + x_3 + x_4 + x_5). 
\]
Each of these terms can be further simplified using Equation (\ref{relationX}). Explicitly, for $i=3$, we have
\[ x_2x_7(x_1+x_2+x_3+x_4+x_5)=x_2x_7(x_3+x_4+x_5).
\]
A long but straightforward computation using the above considerations, Equation (\ref{w3formula}), and the relations given in Equations (\ref{relationX}) and (\ref{relationY}) shows that there are exactly three terms (they are $x_2x_5x_7$, $x_5x_7x_2$, and $x_6x_7x_2$) that square to $x_1x_2x_4 x_5 x_6x_7$. This is enough to conclude that $w_3(Y(A))^2 \neq 0$ as we are working in $H^*(Y(A); \Z/2\Z)$. As mentioned above, it follows from Proposition \ref{propAboutNonEqualTorsion} that  
\[ |T(K_*(Y)| < |\bigoplus_{i} T(H_{*+2i}(Y))|. \]
This completes the proof for $d=9$.

For $d>9$, there are a few options to generalize the construction used above. One can add zeros to the matrix $A$ above or again take a matrix with one's along the superdiagonal except for the last entry and in the last column except for the last two entries. Explicitly, for $d=10$, one can take
\[ \left[ \begin{array}{cccccccccc}0 & 1 & 0 & 0 & 0 & 0 & 0 & 0 & 1 & 0 \\ 0 & 0 & 1 & 0 & 0 & 0 & 0 & 0 & 1 & 0 \\ 0 & 0 & 0 & 1 & 0 & 0 & 0 & 0 & 1 & 0 \\ 0 & 0 & 0 & 0 & 1 & 0 & 0 & 0 & 1 & 0 \\ 0 & 0 & 0 & 0 & 0 & 1 & 0 & 0 & 1 & 0 \\ 0 & 0 & 0 & 0 & 0 & 0 & 1 & 0 & 1 & 0 \\ 0 & 0 & 0 & 0 & 0 & 0 & 0 & 1 & 1 & 0 \\ 0 & 0 & 0 & 0 & 0 & 0 & 0 & 0 & 0 & 0 \\ 0 & 0 & 0 & 0 & 0 & 0 & 0 & 0 & 0 & 0 \\ 0 & 0 & 0 & 0 & 0 & 0 & 0 & 0 & 0 & 0  \end{array} \right] \]
or
\[  \left[ \begin{array}{cccccccccc}0 & 1 & 0 & 0 & 0 & 0 & 0 & 0 & 0 & 1 \\ 0 & 0 & 1 & 0 & 0 & 0 & 0 & 0 & 0 & 1 \\ 0 & 0 & 0 & 1 & 0 & 0 & 0 & 0 & 0 & 1 \\ 0 & 0 & 0 & 0 & 1 & 0 & 0 & 0 & 0 & 1 \\ 0 & 0 & 0 & 0 & 0 & 1 & 0 & 0 & 0 & 1 \\ 0 & 0 & 0 & 0 & 0 & 0 & 1 & 0 & 0 & 1 \\ 0 & 0 & 0 & 0 & 0 & 0 & 0 & 1 & 0 & 1 \\ 0 & 0 & 0 & 0 & 0 & 0 & 0 & 0 & 1 & 1 \\ 0 & 0 & 0 & 0 & 0 & 0 & 0 & 0 & 0 & 0 \\ 0 & 0 & 0 & 0 & 0 & 0 & 0 & 0 & 0 & 0  \end{array} \right] \]
This completes the proof of Theorem \ref{existDimNine}.

As mentioned above, if $Y$ satisfies the conclusion of Theorem \ref{existDimNine}, then by Corollary \ref{mainCor} we have a counterexample to the HK-conjecture that is principal. Thus, we have counterexamples and can take the dimension of the relevant flat manifold to be any integer greater than or equal to nine.


\end{document}